\documentclass[letterpaper, 10 pt, conference]{ieeeconf}  

\IEEEoverridecommandlockouts                              

\usepackage{graphics} 
\usepackage{epsfig} 
\usepackage{mathptmx} 
\usepackage{times} 
\usepackage{amsmath} 
\usepackage{amssymb}  
\usepackage{algorithm}
\usepackage{algorithmic}
\usepackage{array}
\usepackage{comment}
\usepackage{enumerate}
\usepackage{graphicx}
\usepackage{tabularx}
\usepackage{fancyhdr}
\usepackage{extramarks}
\usepackage{lipsum}
\usepackage{color}
\usepackage{graphicx}
\usepackage{setspace}
\usepackage{subfig}
\usepackage{tabularx}

\newtheorem{theorem}{Theorem}
\newtheorem{lemma}{Lemma}

\def\qed{ \rule{.08in}{.08in}}

\def\qed{ \rule{.08in}{.08in}}
\DeclareMathAlphabet{\mathcal}{OMS}{cmsy}{m}{n}

\title{\LARGE \bf On a Modified DeGroot-Friedkin Model of Opinion Dynamics
}

\author{Zhi Xu, Ji Liu, and Tamer Ba\c{s}ar
\thanks{Zhi Xu, Ji Liu, and Tamer Ba\c{s}ar are with the Coordinated Science Laboratory, Department of Electrical and Computer Engineering, University of Illinois at Urbana-Champaign, USA
        ({\tt\small \{zhixu2,jiliu,basar1\}@illinois.edu}).
This research was supported in part by the U.S. Air Force Office of Scientific Research (AFOSR) MURI grant FA9550-10-1-0573.}
}

\begin{document}

\maketitle
\thispagestyle{empty}
\pagestyle{empty}

\begin{abstract}
This paper studies the opinion dynamics that result when individuals consecutively discuss a sequence of issues. Specifically, we study how individuals' self-confidence levels evolve via a reflected appraisal mechanism. Motivated by the DeGroot-Friedkin model, we propose a Modified DeGroot-Friedkin model which allows individuals to update their self-confidence levels by only interacting with their neighbors and in particular, the modified model allows the update of self-confidence levels to take place in finite time without waiting for the opinion process to reach a consensus on any particular issue.
We study properties of this Modified DeGroot-Friedkin model and compare the associated equilibria and stability with those of the original DeGroot-Friedkin model.
Specifically, for the case when the interaction matrix is doubly stochastic, we show that for the modified model, the vector of individuals' self-confidence levels converges to a unique nontrivial equilibrium which for each individual is equal to ${1\over n}$, where $n$ is the number of individuals. This implies that eventually individuals reach a democratic state.
\end{abstract}

\section{Introduction}
Over the years, the advancement in information technology has enabled individuals to be more closely connected and the rapid expansion of online social networks has provided huge amount of data available for analysis regarding how individuals interact over networks. Consequently, much research attention has been drawn to understand how an individual's opinion evolves over time, in particular, how to model the underlying process of opinion formation \cite{DeGroot,Johnsen,Krause}.
There has been increasing interest in developing models of opinion dynamics to capture individuals' interaction, and some of these as relevant to the topic of this paper will be mentioned later. In the literature,
two main approaches have been adopted on how each individual updates her opinion:
probabilistic \cite{Acemoglu,yildiz} and deterministic \cite{DeGroot,Johnsen,Krause}.

In social sciences, quite a few models have been proposed for opinion dynamics.
Notable among them are the three classical models, namely, the DeGroot model \cite{DeGroot}, the Friedkin-Johnsen model \cite{Johnsen}, and the Krause model \cite{Krause}. In the DeGroot model, each individual has a fixed set of neighbors and the local interaction is captured by taking the convex combination of her own opinion and the opinions of her neighbors at each time step. The model can be extended naturally to the case in which the neighbor sets change over time. The Friedkin-Johnsen model is a variation of the DeGroot model in which each individual is assumed to adhere to her initial opinion to a certain degree, which brings in some level of stubbornness. The Krause model defines the neighbor sets in a different way. Each individual takes those individuals whose opinions differ from her by no more than a certain confidence level as her neighbors. It turns out that the Krause model is nonlinear, while the first two models lead to linear opinion updates.

Some recent works have extended the classical models to include more variations. For example, the presence of stubborn individuals has received increasing attention \cite{Acemoglu,yildiz,srikant}.
In \cite{Acemoglu} and \cite{yildiz}, the effects of stubborn individuals who never update their opinions are investigated in a randomized gossiping process. In \cite{srikant}, the opinion formation process is regarded as a local interaction game
and the concept of the stubbornness of an individual regarding her initial opinion is introduced. Other works have extended the Krause model \cite{John} or utilized the idea of confidence level \cite{Bullo2,Bullo3,Bullo6}. The work of \cite{John} introduces and studies a variation of the Krause model, which involves a continuum of agents. In contrast with the Krause model, the neighbors of an individual in \cite{Bullo2} are defined to be those individuals whose influence range contains this individual. The works of \cite{Bullo3,Bullo6} bring exogenous factors, such as the influence of media, into the model and each individual updates her opinion via the opinions of the population inside the individual's confidence range and the information from an exogenous input in that range.

A particularly interesting recent work, which motivates this paper, is the DeGroot-Friedkin model proposed by Jia et al. \cite{Bullo1,Bullo5}. The DeGroot-Friedkin model in \cite{Bullo1,Bullo5} contains two stages and studies the evolution of self-confidence, i.e., how confident an individual is for her opinions on a sequence of issues. In the first stage, individuals update their opinions for a particular issue according to the classical DeGroot model, and in the second stage, the self-confidence for the next issue is governed by the reflected appraisal mechanism studied in \cite{friedkin_reflected,reflected}. Reflected appraisal mechanism, in simple words, describes the phenomenon that individuals' self-appraisals on some dimension (e.g., self-confidence, self-esteem) are influenced by the appraisals of other individuals on them. An extended DeGroot-Friedkin model which includes stubborn individuals has been investigated in \cite{Bullo4}.


Motivated by the original model, we propose a Modified DeGroot-Friedkin model in this paper. For the Modified DeGroot-Friedkin model, we implement the ideas in the original DeGroot-Friedkin model in a distributed way to consider the situation when the process of self-confidence updates takes place within finite time steps for each issue. When the updates take place after infinite time steps, i.e., after the estimated values converge, the original model is exactly recovered (see Section \ref{sec_onestep} for details). Specifically, this paper studies the case when self-confidence is updated after every discussion of an issue. A complete study for the modified model when the interaction matrix is doubly stochastic is provided. We show that the self-confidence vector asymptotically converges to the equal-weights vector $\frac{1}{n}\mathbf{1}$.

The rest of the paper is organized as follows: in Section \ref{sec_DeGroot}, we review the original DeGroot-Friedkin model, and then in Section \ref{sec_onestep}, we discuss some motivations and introduce the Modified DeGroot-Friedkin model. In Section \ref{sec_doubly}, we present the equilibria and stability analysis for the modified model in terms of the doubly stochastic interaction matrix. Finally, we conclude and discuss some future directions in Section \ref{sec_conclusion}.

\textbf{Notation.} All vectors are assumed to be column vectors. For any positive integer n, we use $[n]$ to denote the set $\{1,2,\dots,n\}$. For any two sets $\mathcal{A}$ and $\mathcal{B}$, we use $\mathcal{A}\setminus\mathcal{B}=\{x|x\in\mathcal{A},x\notin\mathcal{B}\}$ to denote the set difference. We use $x'$ to denote the transpose of a vector, and $A'$ to denote the transpose of a matrix. We define $\mathbf{1}$ to be the vector with all entries equal to one in Euclidean space $\mathbb{R}^n$. We use $e_i$ to denote the vector with $1$ in the $i$th entry and $0$ for all other entries. For any two real vectors $x,\:y\in \mathbb{R}^n$, we use $x\geq y$ to denote $x_i\geq y_i$, for all $i\in[n]$, and $x>y$ to denote $x_i>y_i$, for all $i\in[n]$. In addition, $I$ denotes the $n\times n$ identity matrix. For a vector $x$, we use diag$(x)$ to represent the diagonal matrix with the $i$th diagonal entry being $x_i$. A stochastic matrix $A$ is a nonnegative matrix with row sum equal to $1$, i.e., $a_{ij}\geq 0$ for all $i$ and $j\in[n]$, and $\sum_{j=1}^n a_{ij}=1$, for all $i\in[n]$. A left stochastic matrix is a nonnegative matrix with column sum equal to $1$, i.e., $a_{ij}\geq 0$ for all $i$ and $j\in[n]$, and $\sum_{i=1}^n a_{ij}=1$, for all $j\in[n]$. A matrix is doubly stochastic if it is both stochastic and left stochastic. Finally, we use $\Delta$ to denote the $n$-simplex, i.e., $\Delta=\{x\in \mathbb{R}^n|x\geq0,\mathbf{1}'x=1\}$.

\section{The DeGroot-Friedkin Model} \label{sec_DeGroot}
\subsection{Opinion dynamics for a single issue}
We consider a social network with $n>1$ individuals labeled from $1$ to $n$. Each individual $i$ is able to communicate with certain other individuals called individual $i$'s neighbors. Neighbor relations are described by a directed graph $\mathbb{G}$. We will call $\mathbb{G}$ the neighbor graph, in which nodes correspond to individuals and directed edges represent the neighbor relations, i.e., $j$ is a neighbor of $i$ if $(i,j)$ is a directed edge. Consider the case when $n$ individuals are discussing a sequence of issues in the network; let us label each issue as $\{0,1,2,3,\dots\}$ with the understanding that issue $s+1$ will be discussed right after issue $s$. For a fixed issue $s \in \{0,1,2,3,\dots\}$, we denote the $i$th individual's opinion for issue $s$ at time $t$ to be $y_i(s,t)\in \mathbb{R}$. For each issue $s$, the update of $y_i(s,t)$ is determined by the DeGroot model
\begin{equation} \label{degroot_component}
y_i(s,t+1)=w_{ii}(s)y_i(s,t)+\sum_{j=1,j \neq i}^nw_{ij}(s)y_j(s,t),
\end{equation}
or in matrix form
\begin{equation} \label{degroot_matrix}
y(s,t+1)=W(s)y(s,t)
\end{equation}
where $W(s)$ is called the influence matrix and assumed to be stochastic.
From (\ref{degroot_component}), the opinion of individual $i$ at time $t+1$ is a convex combination of all the individuals' opinions at the previous time $t$. We define $w_{ii}(s)$ to be the self-confidence of $i$th individual, i.e., to what extent the $i$th individual adheres to his opinion on issue $s$ at previous time or how confident the $i$th individual is for his opinion on issue $s$ at previous time. Correspondingly, off-diagonal entry $w_{ij}(s)$ determines to what extent, the $i$th individual's opinion will be affected by others.

The work of \cite{Bullo1,Bullo5} studied the evolution of self-confidence $w_{ii}(s)$ for a sequence of issues. For simplicity, let $x_i(s)=w_{ii}(s)$, and we call the vector $x(s)$ as the self-confidence vector for issue $s$. Since $W(s)$ is assumed to be stochastic, $1-x_i(s)$ is then the total weight that individual $i$ assigns to her neighbors. The DeGroot-Friedkin model in \cite{Bullo1,Bullo5} decomposes $w_{ij}(s)$ as
\begin{equation} \label{decomposition}
w_{ij}(s)=(1-x_i(s))c_{ij}
\end{equation}
 Let $C=[c_{ij}]$ be the matrix with entries equal to $c_{ij}$. Matrix $C$ is compliant with the neighbor graph $\mathbb{G}$, and $\mathbb{G}$ is assumed to be strongly connected with no self-loops and fixed across issues. So, $c_{ii}=0$ and $C$ is irreducible. Since $W(s)$ is assumed to be stochastic, from the decomposition (\ref{decomposition}), one can see that the matrix $C$ will be stochastic. We will call $C$ the relative interaction matrix and $c_{ij}$ is correspondingly the relative interpersonal weight that the $i$th individual assigns to her neighbors.

In summary, the final dynamics for a single issue  $s$ is
\begin{equation*}
y_i(s,t+1)=x_i(s)y_i(s,t)+\sum_{j=1,j \neq i}^n(1-x_i(s))c_{ij}y_j(s,t),
\end{equation*}
or in matrix form
\begin{equation}
y(s,t+1)=W(x(s))y(s,t)
\end{equation}
where $W(x(s))=\text{diag}(x(s))+(I-\text{diag}(x(s)))C$.
\subsection{Evolution of self-confidence across a sequence of issues}
We use $u(x(s))$ to denote the normalized (i.e., $\mathbf{1}'u(x(s))=1$) left eigenvector of the influence matrix $W(x(s))$ associated with the eigenvalue 1. From Perron-Frobenius theorem, $u(x(s))>0$ and is unique. We will call $u(x(s))$ the dominant left eigenvector.

It is well known that for the DeGroot model, the limit of the opinions for each issue $s$ is:
\begin{equation} \label{opinion_converge}
\lim_{t \rightarrow \infty}y(s,t)=\lim_{t \rightarrow \infty}W(x(s))^ty(s,0)=u(x(s))'y(s,0)\mathbf{1}
\end{equation}
Therefore, the individuals' opinions for issue $s$ converge to a convex combination of their initial opinions on issue $s$ and the coefficients $u(x(s))$ describe how much each individual contributes to the final opinions. In other words, $u_i(x(s))$ can be regarded as the social power for individual $i$ in determining the final outcomes for a particular issue $s$.

For a sequence of issues $s\in \{0,1,2,3,\dots\}$, the reflected appraisal mechanism introduced in \cite{friedkin_reflected} is to let $x_i(s+1)=u_i(x(s))$. The underlying rationale is that as described in the introduction, the reflected appraisal mechanism captures the property  that individuals' self-appraisals on some dimension (in this case, the self-confidence for issue $s+1$) are affected by the appraisals of other individuals on them. Note that $u_i(x(s))$ represents the social power of individual $i$ in determining the outcomes of issue $s$. If individual $i$ has larger social power, it is very likely that he will be more confident on his own opinion when discussing the next issue $s+1$. Note that for issue $s\geq 1$, the self-confidence vector $x(s)$ necessarily takes value inside $\Delta$ from the update, so it is assumed in \cite{Bullo1,Bullo5} that the self-confidence vector is in $\Delta$ for all issues $s \in \{0,1,2,3,\dots\}$.

Finally, the DeGroot-Friedkin model is
\begin{equation} \label{reflected}
x(s+1)=u(x(s))
\end{equation}
where $u(x(s)) \in \Delta$ and is the dominant left eigenvector of the influence matrix
\begin{equation} \label{demodel}
W(x(s))=\text{diag}(x(s))+(I-\text{diag}(x(s)))C
\end{equation}
The interaction matrix $C$ is assumed to be stochastic and irreducible with diagonal entries being zero.
\begin{theorem}\label{the1}
\cite[Theorem 4.1]{Bullo1} \emph{Suppose that $n\geq 3$ and all $n$ individuals adhere to the DeGroot-Friedkin model defined by (\ref{reflected}) and (\ref{demodel}). Suppose that the underlying neighbor graph $\mathbb{G}$ is not a star\footnotemark[1]. Then, the following statements are true:\\
(1). (\textbf{Equilibria}) The set of fixed points for $x(s+1)=u(x(s))$ is $\{e_1,...,e_n,x^*\}$, where $x^*$ lies in the interior of the simplex $\Delta$.\\
(2). (\textbf{Stability}) For all initial conditions $x(0)\in \Delta \setminus \{e_1,...,e_n\}$, the self-confidence $x(s)$ converges to the equilibrium configuration $x^*$ as $s \rightarrow \infty$.\\
(3). If, in addition, the interaction matrix $C$ is doubly stochastic, then $x^*=\frac{1}{n}\mathbf{1}$ and $x(s)\rightarrow \frac{1}{n}\mathbf{1} $ as $s \rightarrow \infty$.}
\end{theorem}
\footnotetext[1]{A directed graph is star if there is a node, called the center node, having directed edges to and from all other nodes, and for every other node, there are directed edges to and from only the center node.}


\section{Modified DeGroot-Friedkin Model}\label{sec_onestep}

As noted in the above section, the process of reflected appraisal for $x(s+1)$ (i.e., (\ref{reflected})) takes place only after opinions $y(s,t)$ of issue $s$ converges as suggested in (\ref{opinion_converge}), which takes many or infinite number of discussions (i.e., time steps). One may be able to know the self-confidence level for the next issue without waiting that long if she can compute the dominant left eigenvector $u(x(s))$. However, $u(x(s))$ requires global information about the network that is usually impossible and for large networks, a distributed update is often more preferable.

In order to answer the above questions, we first look at a distributed update scheme for self-confidence vector which was proposed in \cite{Bullo1}. Assume that for each individual, she estimates her social power $u_i(x(s))$ for issue $s$ along the time of discussions, and  let us denote the perceived social power for issue $s$ at time $t$ as $p_i(s,t)$. Further, assume that the $i$th individual knows the exact interpersonal weight her neighbors assign to her, i.e., $w_{ji}(s)$ for all $j$ in the neighbors of $i$. Then, every individual updates her perceived self-confidence $p_i(s,t)$ for issue $s$ according to
\begin{equation} \label{distri}
p_i(s,t+1)=w_{ii}(s)p_i(s,t)+\sum_{j=1,j\neq i}^nw_{ji}(s)p_j(s,t)
\end{equation}
which in matrix form is $p(s,t+1)=W(s)'p(s,t)$. We know that $\lim_{t\rightarrow \infty}p(s,t)=u(x(s))$ for all initial states $p(s,0)$ such that $\mathbf{1}'p(s,0)=1$.

However, in order to update $x(s+1)$, we still need to wait for a sufficiently long time for the perceived self-confidence $p(s,t)$ to converge. In order to simultaneously achieve a distributed and finite time update for the self-confidence levels, in this paper, we propose the following model, which is based on the distributed update (\ref{distri}).

Recall that $x_i(s)=w_{ii}(s)$, $w_{ji}(s)=(1-x_j(s))c_{ji}$ and $C$ is stochastic with diagonal entries being zero. Then,  (\ref{distri}) is equivalent to
\begin{equation} \label{modify_distri}
p_i(s,t+1)=x_i(s)p_i(s,t)+\sum_{j=1}^n(1-x_j(s))c_{ji}p_j(s,t)
\end{equation}
Since we want to update $x(s+1)$ in finite time steps instead of waiting for the update of opinions $y(s,t)$ to converge, a straightforward modification is that we update $x(s+1)$ in finite time steps according to (\ref{modify_distri}), i.e.,
\begin{equation} \label{finite_step}
x_i(s+1)=p_i(s,T)
\end{equation}
for some finite time steps $T$. This means that for each issue, after $T$ times discussions, individuals will update their self-confidence levels for the next issue based on the discussions. This model reasonably captures the real scenario in that when we are discussing a sequence of issues with others, we only discuss it for limited times and the opinions need not necessarily converge. Note that if $T$ goes to infinity, the original DeGroot-Friedkin model is exactly recovered. Therefore, in this sense, the above model is a generalization of the original model. In this paper, we focus on the one-step case when $T=1$. Then, (\ref{modify_distri}) becomes
\begin{equation} \label{onestep_component}
x_i(s+1)=x_i(s)x_i(s)+\sum_{j=1}^n(1-x_j(s))c_{ji}x_j(s)
\end{equation}

From (\ref{onestep_component}), to update self-confidence, one only needs to know the self-confidence levels of her neighbors and the interpersonal weight $c_{ji}$ from her neighbors. We will refer to (\ref{onestep_component}) as the Modified DeGroot-Friedkin model. We make the same assumption about the neighbor graph $\mathbb{G}$ as in the DeGroot-Friedkin model, namely, the neighbor graph $\mathbb{G}$ is strongly connected with no self-loops and fixed over time. Note that we drop the argument $t$ in (\ref{onestep_component}) since we update $x(s)$ in every time step.

In summary, the Modified DeGroot-Friedkin model in matrix form is
\begin{equation} \label{onestep_matrix}
x(s+1)=C'x(s)+X(s)x(s)-C'X(s)x(s)
\end{equation}
where $X(s)=\text{diag}(x(s))$ is a diagonal matrix. The relative interaction matrix $C$ is assumed to be irreducible and stochastic with diagonal entries being zero.
\begin{lemma} \label{lemma1}
\emph{Suppose that all $n$ individuals adhere to the Modified DeGroot-Friedkin model defined by (\ref{onestep_matrix}). Suppose that the interaction matrix $C$ is stochastic and irreducible with diagonal entries being zero. Then, the sum of self-confidence levels is constant, i.e., $\sum_{i=1}^nx_i(s)=\sum_{i=1}^nx_i(0)$, for all $s\in\{0,1,2,\dots\}$.}
\end{lemma}
\begin{proof}
Multiply both sides of (\ref{onestep_matrix}) from the left by $\mathbf{1}'$.
Then, the left hand side is $\mathbf{1}'x(s+1)=\sum_{i=1}^nx_i(s+1)$. Recall that $C$ is stochastic, so $\mathbf{1}'C'=\mathbf{1}'$. Then, the right hand side is
\begin{equation*}
\begin{split}
  & \quad\mathbf{1}'C'x(s)+\mathbf{1}'X(s)x(s)-\mathbf{1}'C'X(s)x(s)\\
     &=\sum_{i=1}^nx_i(s)+\sum_{i=1}^nx_i(s)^2-\sum_{i=1}^nx_i(s)^2 =\sum_{i=1}^nx_i(s)
\end{split}
\end{equation*}
Therefore, $\sum_{i=1}^nx_i(s+1)=\sum_{i=1}^nx_i(s)=\sum_{i=1}^nx_i(0)$.
\end{proof}

\textbf{Remark:} We will make the same assumption on the initial conditions as in the original DeGroot-Friedkin model, namely, $x(0)\in\Delta$. One reason for doing this is that from (\ref{onestep_component}), we can see if $x_i(s)\in [0,1]$, one can ensure that $x_i(s+1)$ is positive since the square term $x_i(s)^2\leq x_i(s)$. Further, by assuming $x(0)\in \Delta$, Lemma \ref{lemma1} ensures that $\sum_{i=1}^nx_i(s)=1$, for all $s\in\{0,1,2,\dots\},$ so that $x_i(s)$ will always be bounded in $[0,1]$. If we instead assume $x_i(0)\in [0,1]$, for all $i\in[n]$, but place no condition on the sum, then $x_i(s)$ will not always be in $[0,1]$ for future updates, and $x_i(s)$ may be unbounded eventually. Simulation results also confirm the unboundedness property. A future direction of research is to relax this assumption, because while it is reasonable to assume that the self-confidence levels are within $[0,1]$, it is less reasonable that the self-confidence levels should sum up to 1, especially in large social networks.
\hfill$\qed$

\begin{lemma} \label{lemma2}
\emph{For all $i\in[n]$, $e_i$ is an equilibrium of the Modified DeGroot-Friedkin model defined by (\ref{onestep_matrix}).}
\end{lemma}
\begin{proof}
Substituting $x(s)=e_i$ into (\ref{onestep_matrix}), it is straightforward to see that the right hand side is $C'e_i+\text{diag}(e_i)e_i-C'\text{diag}(e_i)e_i=e_i$.
\end{proof}

Now, we show some properties regarding the Modified DeGroot-Friedkin model which will be used in proving other results later.
\begin{lemma} \label{lemma 3}
\emph{Suppose that $n\geq 3$ and all $n$ individuals adhere to the Modified DeGroot-Friedkin model defined by (\ref{onestep_matrix}). Suppose that the relative interaction matrix $C$ is stochastic and irreducible with diagonal entries being zero. If $x(0)\in \Delta_n \setminus \{e_1,...,e_n\}$, then the following properties hold:
\begin{enumerate}[(i)]
\item If $x(0)>0$, then $x(s)>0$, for all $s\geq 1$.
\item If $x(s)>0$, then matrix $T(s)=C'+X(s)-C'X(s)$ is irreducible.
\item Let $m$ be the number of zero entires in $x(0)$. If $x(0)\geq 0$, then after finite steps $\tau\leq m$, $x(\tau)>0$.
\end{enumerate}}
\end{lemma}
\begin{proof}
\emph{(i)} The dynamics (\ref{onestep_matrix}) can be written as:
\begin{equation*}
x(s+1)=C'(I-X(s))x(s)+X(s)x(s)
\end{equation*}
If $x(0)>0$, then the vector $X(0)x(0)=(x_1^2(0),\dots,x_n^2(0))'>0$. In addition, $(I-X(0))$ is a diagonal matrix with positive diagonal entries when $x(0)>0$. Since $C'$ is a nonnegative matrix, $C'(I-X(0))x(0)\geq0$. Therefore, we conclude that $x(1)=C'(I-X(0))x(0)+X(0)x(0)>0$. Applying the same argument repeatedly, we conclude that $x(s)>0$, for all $s\geq 1$.

\emph{(ii) }Since $C$ is assumed to be irreducible, $C'$ is also irreducible. $T(s)=C'+X(s)-C'X(s)=C'(I-X(s))+X(s)$. If $x(s)>0$, then $(I-X(s))$ is a diagonal matrix with positive diagonal entries. When nonnegative matrix $C'$ is multiplied by a diagonal matrix with positive diagonal entries, if $c'_{ij}=0$, then $(C'(I-X(s)))_{ij}=0$; if $c'_{ij}>0$, then $(C'(I-X(s)))_{ij}>0$. The structure of $C'$ does not change, so $C'(I-X(s))$ is also irreducible with diagonal entries being zero. It is well known that a matrix is irreducible if and only if the underlying directed graph represented by the matrix is strongly connected. Strong connectedness means that starting from any node, one can find a directed path to any other node in the graph. Therefore, $C'(I-X(s))+X(s)$ only adds self-loops in the underlying directed graph represented by $C'(I-X(s))$, which will not affect the connectivity. So, we conclude that $T(s)$ is irreducible.

\emph{(iii) }The dynamics (\ref{onestep_matrix}) can be written as:
\begin{equation} \label{lemma3_dynamic}
\begin{split}
x(s+1)&=C'(I-X(s))x(s)+X(s)x(s)\\
      &=Pb(s)+a(s)
\end{split}
\end{equation}
where for simplicity, we use $P=C'$, vector $a(s)=(x_1^2(s),\dots,x_n^2(s))'$, and vector $b(s)=((1-x_1(s))x_1(s),\dots,(1-x_n(s))x_n(s))'$. So, $P$ is irreducible. Suppose that the number of zero entries in $x(0)$ is $m$, without loss of generality, let us assume $x_i(0)=0$, for $i\in \{1,2,\dots,m\}$, and $x_i(0)>0$ for $i \in\{m+1,m+2,\dots,n\}$.

We first claim that the number of zero entries can not increase after one update, that is, the number of zero entries in $x(1)$ is at most $m$. From the structures of $a(s)$ and $b(s)$, we see that $a_i(0)=0$, $b_i(0)=0$ for $i\in\{1,2,\dots,m\}$ and $a_i(0)>0$, $b_i(0)>0$ for $i \in\{m+1,m+2,\dots,n\}$. Since $P$ is a nonnegative matrix, from (\ref{lemma3_dynamic}), we see that necessarily, $x_i(1)>0$ for all $i\in\{m+1,m+2,\dots,n\}$, so the number of zero entries in $x(1)$ is at most $m$.

We then prove by contradiction that the number of zero entries must decrease. We first note that $P$ is irreducible so that the underlying directed graph represented by $P$ is strongly connected. We use the convention that if $p_{ij}>0$, then there is a directed edge from node $i$ to node $j$. Now, suppose that the number of zero entries in $x(1)$ is still $m$, then necessarily, $x_i(1)=0$ for $i\in\{1,2,\dots,m\}$. From the dynamics (\ref{lemma3_dynamic}), this is possible if and only if
\begin{equation}\label{lemma3_condition}
p_{ij}=0,\:\text{for all}\:i\in\{1,2,\dots,m\} \:\text{and}\:j\in\{m+1,m+2,\dots,n\}
\end{equation}
Transforming the condition (\ref{lemma3_condition}) into the underlying directed graph of $P$, it means that we can divide the $n$ nodes into two components. Component $1$ contains nodes $\{1,2,\dots,m\}$ and component $2$ contains nodes $\{m+1,m+2,\dots,n\}$. Further, there are no directed edges from component $1$ to component $2$. This is contradicted by the assumption that $P$ is irreducible. So, the number of zero entries in $x(1)$ must decrease.

Finally, since the number of zero entries decreases in $x(1)$, assume now that there are $k<m$ zero entries in $x(1)$ and let $\mathcal{Z}_0=\{i|x_i(1)=0,0\leq i\leq n\}$ and $\mathcal{Z}_1=\{i|x_i(1)>0,0\leq i\leq n\}$. Then, $|\mathcal{Z}_0|=k$ and $|\mathcal{Z}_1|=n-k$, where $|.|$ denotes cardinality. Applying the same reasoning, $x_i(2)>0$ for $i\in \mathcal{Z}_1$. Again from the dynamics (\ref{lemma3_dynamic}), the number of zero entries in $x(2)$ is $k$ if and only if
\begin{equation*}
p_{ij}=0,\:\text{for all}\:i\in \mathcal{Z}_0 \:\text{and}\:j\in \mathcal{Z}_1
\end{equation*}
which as before contradicts the assumption that $P$ is irreducible if we divide the nodes into two components with one component containing nodes $i\in \mathcal{Z}_0$ and one component containing nodes $j\in \mathcal{Z}_1$. So, the number of zero entries in $x(2)$ must decrease. Continuing the same reasoning, we conclude that for each update the number of zero entries in $x(s)$ must strictly decrease. Since the number of zero entries in $x(0)$ is $m$, after at most $m$ steps, all the entries of $x(s)$ will become positive.
\end{proof}

\section{Equilibria and Stability For The Case Of Doubly Stochastic $C$}\label{sec_doubly}
In this section, we will further explore equilibria and stability properties for the case when the relative interaction matrix $C$ is doubly stochastic and we will compare our results with the original DeGroot-Fredkin model

We first state the following Perron-Frobenius theorem for irreducible Metzler matrices, which is used for proving the next theorem. A Metzler matrix is a matrix whose off-diagonal entries are all nonnegative.
\begin{lemma}\label{lemma_pf}
\cite[Theorem 17]{positive} \emph{Let $M\in \mathbb{R}^{n\times n}$ be an irreducible Metzler matrix and $\sigma(M)$ be the set of eigenvalues. Then,\\
(i) $\mu(M)=\max_{\lambda\in\sigma(M)}Re(\lambda)$ is an algebraically simple eigenvalue of $M$.\\
(ii) Let $v_F$ be such that $Mv_F=\mu(M)v_F$. Then, $v_F$ is unique (up to scalar multiple) and $v_F> 0$.\\
(iii) If $v\geq 0$ but $v\neq 0$ is an eigenvector of $M$, then $Mv=\mu(M)v$, and hence, $v$ is a scalar multiple of $v_F$.}
\end{lemma}

The equilibria of the Modified DeGroot-Friedkin model are characterized in the following theorem.

\begin{theorem}\label{thm3}
\emph{\textbf{(Equilibria)} Suppose that $n\geq 3$ and all $n$ individuals adhere to the Modified DeGroot-Friedkin model defined by (\ref{onestep_matrix}). Suppose that the relative interaction matrix $C$ is doubly stochastic and irreducible with diagonal entries being zero, then besides the trivial equilibrium points $e_1,\dots,e_n$, there exists a unique nontrivial equilibrium $x^*=\frac{1}{n}\mathbf{1}$}.
\end{theorem}
\begin{proof}
We first establish the fact that $x^*=\frac{1}{n}\mathbf{1}$ is an equilibrium for the Modified DeGroot-Friedkin model. Substituting $x(s)=\frac{1}{n}\mathbf{1}$ into (\ref{onestep_matrix}), the right hand side becomes
\begin{equation*}
\begin{split}
  &\quad C'x(s)+X(s)x(s)-C'X(s)x(s)\\
     &=C'\left(\frac{1}{n}\mathbf{1}\right)+\frac{1}{n}I\left(\frac{1}{n}\mathbf{1}\right)-C'\frac{1}{n}I\left(\frac{1}{n}\mathbf{1}\right)\\
     & =\frac{1}{n}\mathbf{1}+\frac{1}{n^2}\mathbf{1}-\frac{1}{n^2}\mathbf{1}=\frac{1}{n}\mathbf{1}
\end{split}
\end{equation*}
The second equality is due to the fact that $C$ is doubly stochastic, i.e., $C'\left(\frac{1}{n}\mathbf{1}\right)=\frac{1}{n}\mathbf{1}$. Therefore, $x=\frac{1}{n}\mathbf{1}$ is an equilibrium for the Modified DeGroot-Friedkin model. We now want to show that $\frac{1}{n}\mathbf{1}$ is unique besides the trivial equilibria. Let $x^*$ be an equilibrium and from (\ref{onestep_matrix}), $x^*$ must satisfy
\begin{equation*}
x^*=C'x^*+X^*x^*-C'X^*x^*
\end{equation*}
or
\begin{equation*}
(C'-I)(x^*-X^*x^*)=0
\end{equation*}
Note that the $i$th entry of vector $x^*-X^*x^*$ is $(x^*-X^*x^*)_i=x^*_i-(x^{*}_i)^2$. As discussed before, since the initial condition $x(0)\in\Delta$ and from Lemma \ref{lemma1}, $x_i(s)$ is necessarily always in [0,1]. As a consequence, $x^*-X^*x^*\geq0$. To apply \emph{(iii)} of Lemma \ref{lemma_pf}, we still need to show that $x^*-X^*x^*\neq0$. Suppose that $x^*-X^*x^*=0$, which means $x^*_i-(x^*_i)^2=0$, for all $i\in[n]$, then $x_i^*=0$ or $x_i^*=1$, for all $i\in[n]$. We know $\sum_{i=1}^nx_i(s)=1$, so $x^*-X^*x^*=0$ means $x^*\in \{e_1,\dots,e_n\}$, which are the trivial equilibria for the system. Therefore, for a nontrivial equilibrium, $x^*-X^*x^*\geq0$ and $x^*-X^*x^*\neq0$. Since $C'-I$ is a Metzler matrix and $C$ is assumed to be irreducible, from \emph{(iii)} of Lemma \ref{lemma_pf}, we conclude that $x^*-X^*x^*$ is unique up to some scalar multiple. We already know that $\frac{1}{n}\mathbf{1}$ is an equilibrium of the system; hence
\begin{equation*}
x^*-X^*x^*=a\left(\frac{1}{n}\mathbf{1}-\frac{1}{n^2}\mathbf{1}\right) \quad\text{for some}\:\:\:a>0
\end{equation*}
or for each component
\begin{equation} \label{theorem2_component}
x^*_i-(x^*_i)^2=a\frac{n-1}{n^2} \quad\text{for all $i\in[n]$ and for some $a>0$}
\end{equation}
Here,
$a>0$ is a result of the fact that $x^*_i\in[0,1]$. Although we have shown $x^*-X^*x^*$ is unique up to some scalar multiple, (\ref{theorem2_component}) is a quadratic equation and $a$ can be an arbitrary positive number. There may be an infinite number of choices for $x^*$ that satisfies (\ref{theorem2_component}). Now, we want to show that under the condition $\sum_{i=1}^nx^*_i=1$, there is only one choice.

For $a$ such that $1-\frac{4a(n-1)}{n^2}\geq 0$ , the quadratic equation (\ref{theorem2_component}) admits two solutions
\begin{equation*}
x_{i1}^*=\frac{1+\sqrt{1-\frac{4a(n-1)}{n^2}}}{2}\quad x_{i2}^*=\frac{1-\sqrt{1-\frac{4a(n-1)}{n^2}}}{2}
\end{equation*}
Note that both solutions are positive.

\emph{(i)} Consider first the case when $x^*_i=x^*_{i2}$, for all $i\in[n]$.
\begin{equation*}
\begin{split}
\sum_{i=1}^nx^*_i =\frac{n-n\sqrt{1-\frac{4a(n-1)}{n^2}}}{2}=1
\end{split}
\end{equation*}
which is equivalent to $a=1$. When $a=1$, the above corresponds to the case when $x^*=\frac{1}{n}\mathbf{1}$.

\emph{(ii)} Now consider the second case, that is $x^*_j=x^*_{j1}$ for some $j$ and $x^*_i=x^*_{i2}$, for all $i\neq j$, $i\in[n]$. Then, we need
\begin{equation*}
\sum_{i=1}^nx^*_i= \frac{n-(n-2)\sqrt{1-\frac{4a(n-1)}{n^2}}}{2}=1
\end{equation*}
which is the same as
\begin{equation*}
n-(n-2)\sqrt{1-\frac{4a(n-1)}{n^2}}=2
\end{equation*}
However, $a>0$ implies $\sqrt{1-\frac{4a(n-1)}{n^2}}<1$, so
\begin{equation} \label{theorem_condition}
n-(n-2)\sqrt{1-\frac{4a(n-1)}{n^2}}>2
\end{equation} and hence this case is not possible.

\emph{(iii)} Now one can show that for the case when $x^*_j=x^*_{j1}$, $x^*_m=x^*_{m1}$ for some $j\neq m$ and $x^*_i=x^*_{i2}$, for all $i\neq j,\:i\neq m$, and $i\in[n]$, we need
\begin{equation*}
n-(n-4)\sqrt{1-\frac{4a(n-1)}{n^2}}=2
\end{equation*}
which is not possible by (\ref{theorem_condition}). Applying the same idea, in general, one can see that when there are $k$ entries of $x^*$ that are equal to $x^*_{i1}$ and the remaining $n-k$ entries are equal to $x^*_{i2}$, we need
\begin{equation*}
n-(n-2k)\sqrt{1-\frac{4a(n-1)}{n^2}}=2
\end{equation*}
which is impossible for $k\geq 1$ by (\ref{theorem_condition}).

Therefore, the only choice under the condition $\sum_{i=1}^n=1$ is the first case, and hence $x^*=\frac{1}{n}\mathbf{1}$ is the only nontrivial equilibrium.
\end{proof}

\textbf{Remark:} The above theorem only considers the case when $n\geq 3$. The case when $n=2$ is interesting, and requires a separate treatment. While $x_1^*=x^*_{i1}$ and $x_2^*=x^*_{i2}$ is still an equilibrium which corresponds to $x^*=(0.5,0.5)'$, actually, for $x_1^*=x^*_{i1}$ and $x_2^*=x^*_{i2}$, $x_1^*+x_2^*=1$ for any $a>0$. Therefore, as long as we choose any $a>0$ that makes the quadratic equation (\ref{theorem2_component}) have solutions in $[0,1]$, $x^*=(x_{i1}^*,x_{i2}^*)'$ is always an equilibrium. The system actually has an infinite number of nontrivial equilibria for the $n=2$ case. Since $n=2$ is too small for almost all social networks, we will focus our attention on cases when $n\geq 3$.
\hfill$\qed$

Next, we will show that if the initial conditions are not the trivial equilibria, the self-confidence vector $x(s)$ converges to the unique nontrivial equilibrium $x^*$.
\begin{theorem} \label{the4}
\emph{\textbf{(Asymptotically stable)} Suppose that $n\geq 3$ and all $n$ individuals adhere to the Modified DeGroot-Friedkin model defined by (\ref{onestep_matrix}). Suppose that the relative interaction matrix $C$ is doubly stochastic and irreducible with diagonal entries being zero. Then, for any initial condition $x(0)\in \Delta \setminus \{e_1,\dots,e_n\}$, $x(s)$ asymptotically converges to the unique nontrivial equilibrium $x^*=\frac{1}{n}\mathbf{1}$ as $s\rightarrow\infty$.}
\end{theorem}

To prove Theorem \ref{the4}, we first state the following lemma regarding the non-expansiveness of the minimal and maximal entries in $x(s)$.
\begin{lemma} \label{lemma4}
\emph{Suppose that $n\geq 3$ and all $n$ individuals adhere to the Modified DeGroot-Friedkin model defined by (\ref{onestep_matrix}). Assume that the relative interaction matrix $C$ is doubly stochastic and irreducible with diagonal entries being zero, and that $x(0)\in \Delta_n \setminus \{e_1,\dots,e_n\}$. Let $x(s)_{min}=\min_{0\leq i\leq n} x_i(s)$ and $x(s)_{max}=\max_{0\leq i\leq n} x_i(s)$. Then, for every update, $x(s)_{min}$ and $x(s)_{max}$ are not expanding, i.e., $x(s+1)_{min}\geq x(s)_{min}$ and $x(s+1)_{max}\leq x(s)_{max}$, for all $s\geq 0$.}
\end{lemma}
\begin{proof}
First, let us consider the following equation:
\begin{equation}\label{lemma4_equation}
\begin{split}
& y_1^{new}=y_1^2+y_2-y_2^2\quad\\
 s.t.\quad & y_1\geq 0,\quad y_2\geq 0,\quad y_1+y_2\leq 1
\end{split}
\end{equation}
Then, we have
\begin{equation*}
\begin{split}
 y_1^{new}-y_2&=y_1^2-y_2^2\\
y_1^{new}-y_1&=(y_2-y_1)-(y_2^2-y_1^2)\\
&=(y_2-y_1)(1-y_2-y_1)
\end{split}
\end{equation*}
Therefore, if $y_1\leq y_2$, we have $y_1\leq y_1^{new}\leq y_2$; if $y_2\leq y_1$, we have $y_2\leq y_1^{new}\leq y_1$. So for the equation (\ref{lemma4_equation}), we obtain the following result:
\begin{equation} \label{lemma4_equationresult}
min(y_1,y_2)\leq y_1^{new}\leq max(y_1,y_2)
\end{equation}

Now, let us consider the component-wise Modified DeGroot-Friedkin model (\ref{onestep_component}).
\begin{equation} \label{lemma4_component}
x_i(s+1)=x_i^2(s)+\sum_{j=1,j\neq i}^n(x_j(s)-x_j^2(s))c_{ji}
\end{equation}
Note that since the diagonal entries of $C$ are zero, so for the second term in (\ref{lemma4_component}), taking summation from $j=1$ to $n$ is the same as taking summation from $j=1,j\neq i$ to $n$.

Since $C$ is doubly stochastic with diagonal entries being zero, $\sum_{j=1,j\neq i}^n c_{ji}=1$. Therefore, $\sum_{j=1,j\neq i}^n(x_j(s)-x_j^2(s))c_{ji}$ is indeed the convex combination of $(x_j(s)-x_j^2(s))$ for $j\neq i$, $j\in[n]$. Hence,
\begin{equation}\label{lemma4_convexcondition}
\begin{split}
 z(i)_{min}  \leq\sum_{j=1,j\neq i}^n(x_j(s)-x_j^2(s))c_{ji} \leq z(i)_{max}
\end{split}
\end{equation}
where $z(i)_{min}=min_{j\neq i,0\leq j\leq n}(x_j(s)-x^2_j(s))$ and $z(i)_{max}=max_{j\neq i,0\leq j\leq n}(x_j(s)-x^2_j(s))$.

Now, let $v(i)_{min}=min_{j\neq i,0\leq j\leq n}\:x_j(s)$ and $v(i)_{max}=max_{j\neq i,0\leq j\leq n}\:x_j(s)$. $v(i)_{min}$ is the minimal value among all $x_j(s)$ excluding $x_i(s)$ and correspondingly $v(i)_{max}$ is the maximal value excluding $x_i(s)$.

We claim that there exists $\bar{x}_i(s)\in[v(i)_{min},v(i)_{max}]$ such that
\begin{equation}\label{lemma4_claim}
\bar{x}_i(s)-\bar{x}_i^2(s)=\sum_{j=1,j\neq i}^n(x_j(s)-x_j^2(s))c_{ji}
\end{equation}
A graphical representation will help understand the claim.
\setlength{\intextsep}{-0.01cm}
\begin{figure}[htbp]
\centering
\includegraphics[width=0.35\textwidth]{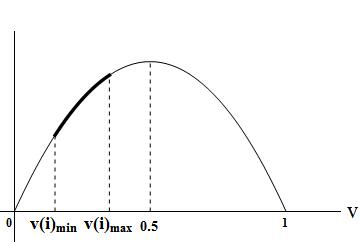}
\caption{Function $f(v)=v-v^2$, for $v(i)_{max}< 0.5$}
\setlength{\belowcaptionskip}{0.03cm}
\includegraphics[width=0.35\textwidth]{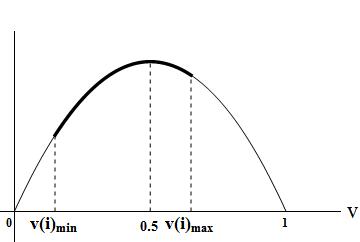}
\caption{Function $f(v)=v-v^2$, for $v(i)_{max}\geq 0.5$}
\end{figure}
Figures $1$ and $2$ are the plots of function $f(v)=v-v^2$. Since $\sum_{i=1}^n x_i=1$, then necessarily $v(i)_{min}\leq 0.5$ and $v(i)_{max}\leq 1-x(i)_{min}$. Function $f(v)$ is strictly increasing for $[0,0.5]$, so $z(i)_{min}=v(i)_{min}-v(i)^2_{min}$.

\emph{Case 1.} $v(i)_{max}< 0.5$.
In this case, $z(i)_{max}=v(i)_{max}-v(i)^2_{max}$. From (\ref{lemma4_convexcondition}), $\sum_{j=1,j\neq i}^n(x_j(s)-x_j^2(s))c_{ji}$ must lie in the bold curve of Figure $1$, so there must exist $\bar{x}_i(s)\in[v(i)_{min},v(i)_{max}]$ such that (\ref{lemma4_claim}) holds.

\emph{Case 2.} $v(i)_{max}\geq 0.5$. If $z(i)_{max}= v(i)_{max}-v(i)^2_{max}$, we apply the same reasong in \emph{Case 1}. If $z(i)_{max}\neq v(i)_{max}-v(i)^2_{max}$, then there exists $x_k<v(i)_{max}$ such that $z(i)_{max}= x_k(s)-x_k^2(s)$. So from (\ref{lemma4_convexcondition}), there must exist $\bar{x}_i(s)\in[v(i)_{min}, x_k(s)]$ such that (\ref{lemma4_claim}) holds. Since $x_k<v(i)_{max}$, from Figure 2, $\sum_{j=1,j\neq i}^n(x_j(s)-x_j^2(s))c_{ji}$ must lie in the bold curve of Figure $2$. Hence, there exists $\bar{x}_i(s)\in[v(i)_{min},v(i)_{max}]$ such that (\ref{lemma4_claim}) holds. This completes the proof of our claim.

Finally, we can write (\ref{lemma4_component}) as
\begin{equation} \label{lemma4_bar}
x_i(s+1)=x_i^2(s)+\bar{x}_i(s)-\bar{x}_i(s)^2
\end{equation}
Further, since $\bar{x}_i(s)<v(i)_{max}$ and $\sum_{i=1}^nx_i(s)=1$, then $x_i(s)+\bar{x}_i(s)\leq 1$ and $x_i(s)\geq 0$, $\bar{x}_i(s)\geq 0$. Comparing (\ref{lemma4_bar}) with (\ref{lemma4_equation}), they have the same structures and conditions. Therefore, from (\ref{lemma4_equationresult}), we conclude:
\begin{equation*}
min(x_i(s),\bar{x}_i(s))\leq x_i(s+1)\leq max(x_i(s),\bar{x}_i(s))
\end{equation*}
Since $\bar{x}_i(s)\in[v(i)_{min},v(i)_{max}]$, we finally get
\begin{equation} \label{lemma4_final}
x(s)_{min}\leq x_i(s+1)\leq x(s)_{max}
\end{equation}
where $x(s)_{min}=min_{1\leq k\leq n} x_k(s)$ and $x(s)_{max}=max_{1\leq k\leq n} x_k(s)$. For the above proof, it can be seen that $i$ can be arbitrarily chosen. Therefore, we conclude from (\ref{lemma4_final}) that
\begin{equation} \label{lemma4_final2}
x(s)_{min}\leq x_i(s+1)\leq x(s)_{max}\quad\text{for all $i\in[n]$}
\end{equation}
which is equivalent to
\begin{equation*}
x(s)_{min}\leq x(s+1)_{min}\quad\text{and}\quad x(s+1)_{max}\leq x(s)_{max}
\end{equation*}
\end{proof}

Based on the above lemma, we provide an even stronger result regarding the evolution of the maximum of $x(s)$.
\begin{lemma} \label{lemma_decrease}
\emph{Suppose that $n\geq 3$ and all $n$ individuals adhere to the Modified DeGroot-Friedkin model defined by (\ref{onestep_matrix}). Assume that the relative interaction matrix $C$ is doubly stochastic and irreducible with diagonal entries being zero. Let $x(s)_{max}=\max_{0\leq i\leq n} x_i(s)$. If $x(s)>0$ and $x(s)\neq \frac{1}{n}\mathbf{1}$, then $x(s)_{max}$ must decrease in at most $n-1$ updates, i.e., $x(s+n-1)_{max}<x(s)_{max}$.}
\end{lemma}

\begin{proof}
First, let us consider a variation of (\ref{lemma4_equation}) in the proof of Lemma \ref{lemma4},
\begin{equation}\label{lemma_decrease_equation}
\begin{split}
& y_1^{new}=y_1^2+y_2-y_2^2\quad\\
 s.t.\quad & y_1> 0,\quad y_2> 0,\quad, y_1\neq y_2,\quad y_1+y_2< 1
\end{split}
\end{equation}
Then, from the same reasoning, we have
\begin{equation} \label{lemma_decrease_equation_result}
min(y_1,y_2)<y_1^{new}<max(y_1,y_2)
\end{equation}
Recall that in the proof of Lemma \ref{lemma4}, we have shown from (\ref{lemma4_claim}) and (\ref{lemma4_bar}) that for any $i\in[n]$, there exists $\bar{x}_i(s)\in[v(i)_{min},v(i)_{max}]$ such that
\begin{equation}
\bar{x}_i(s)-\bar{x}_i^2(s)=\sum_{j=1,j\neq i}^n(x_j(s)-x_j^2(s))c_{ji}
\end{equation}
and
\begin{equation} \label{lemma_decrease_bar}
x_i(s+1)=x_i^2(s)+\bar{x}_i(s)-\bar{x}_i(s)^2
\end{equation}
Note that by assuming $x(s)>0$, we immediately have $x_i(s)>0$, $\bar{x}_i(s)>0$, and $x_i(s)+\bar{x}_i(s)<1$. Comparing with (\ref{lemma_decrease_equation}), we notice that if we can show that for any $i\in[n]$, $x_i(s)\neq \bar{x}_i(s)$, then we will have the desirable property that the maximum of $x(s)$ is decreasing.

To this end, let us define two sets first. Let $\mathcal{M}$ be the set containing the indices of the maximum elements in $x(s)$, i.e.,
\begin{equation*}
\mathcal{M}=\{i|i\in[n],\:x_i(s)=x(s)_{max}\},
\end{equation*}
and let $\mathcal{Q}$ be the set containing the indices of all the elements in $x(s)$ such that $x_i(s)=\bar{x}_i(s)$, i.e.,
\begin{equation*}
\mathcal{Q}=\{i|i\in[n],\:x_i(s)=\bar{x}_i(s)\}.
\end{equation*}
We will use $|.|$ to denote the cardinality of a set.

Note that for any $i\in\mathcal{M}$, \emph{Case 2} in the proof of Lemma \ref{lemma4} (i.e., $v(i)_{max}\geq 0.5$) cannot happen. This is because $v(i)_{max}\leq x_i(s)$ for any $i\in\mathcal{M}$. From $x(s)>0$ and $x(s)\in\Delta$, we then have $v(i)_{max}<0.5$. Therefore, we only need to consider \emph{Case 1} (i.e., $v(i)_{max}<0.5$). From Fig. 1, we can then conclude that $\bar{x}_i(s)$ is unique because $f(v)=v-v^2$ is a strictly increasing function in the interval (0,0.5). In addition, note that for any $i\in\mathcal{Q}$, we have $x_i(s+1)=x_i(s)$ and any $i\notin\mathcal{Q}$, we have
\begin{equation*}
x(s)_{min}<x_i(s+1)<x(s)_{max}.
\end{equation*}

Now, to prove Lemma \ref{lemma_decrease}, we consider the following scenarios:

\emph{Scenario 1: $|\mathcal{M}|=1$.}

In this scenario, there is only one maximum element in $x(s)$. Without loss of generality, let $x_1(s)=x(s)_{max}$. We claim that $1\notin\mathcal{Q}$. This is because $\bar{x}_1(s)\in[v(1)_{min},v(1)_{max}]$ and $v(1)_{max}<x_1(s)$. Therefore, we have
\begin{equation*}
x(s)_{min}<x_1(s+1)<x(s)_{max}=x_1(s)
\end{equation*}
For any $i\in\{2,3,\dots,n\}$, if $i\in\mathcal{Q}$, then $x_i(s+1)=x_i(s)<x_1(s)$. If $i\notin\mathcal{Q}$, then
\begin{equation*}
x(s)_{min}<x_i(s+1)<x(s)_{max}=x_1(s)
\end{equation*}
Therefore, we conclude that $x(s+1)_{max}<x(s)_{max}$.

\emph{Scenario 2: $|\mathcal{M}|=2$.}

In this scenario, there are two maximum elements in $x(s)$. Again, without loss of generality, let $x_1(s)=x_2(s)=x(s)_{max}$ and we have $x_1(s)<0.5$ and $x_2(s)<0.5$. Now, we can identify three cases.

\emph{Scenario 2.A: $1\notin\mathcal{Q}$ and $2\notin\mathcal{Q}$.}

Following exactly the same analysis in \emph{Scenario 1}, we conclude that $x(s+1)_{max}<x(s)_{max}$.

\emph{Scenario 2.B: $1\in\mathcal{Q}$ and $2\notin\mathcal{Q}$.}

In this case, we have
\begin{equation*}
x(s)_{min}<x_2(s+1)<x(s)_{max}.
\end{equation*}

Let us recall that:
\begin{equation} \label{lemma_decrease_claim}
\bar{x}_i(s)-\bar{x}_i^2(s)=\sum_{j=1,j\neq i}^n(x_j(s)-x_j^2(s))c_{ji}
\end{equation}
We note that $v(1)_{max}=x_2(s)$ and $f(v)=v-v^2$ is a strictly increasing function in interval (0,0.5) (Fig. 1). Therefore, $1\in\mathcal{Q}$ implies that $c_{21}=1$ and $c_{j1}=0$ for any $j\in[n]$ and $j\neq 2$, as $C$ is assumed to be doubly stochastic with zero diagonal elements. From the component-wise Modified DeGroot-Friedkin model (\ref{onestep_component}), we have
\begin{equation}\label{lemma_decrease_sce2b}
x_1(s+1)=x_1^2(s)+x_2(s)-x_2^2(s)
\end{equation}
Therefore, we have $x_1(s+1)=x_1(s)$, $x_2(s+1)<x_1(s)$, and $x_j(s+1)<x_1(s)$ for any $j\in\{2,3,\dots,n\}$. $x_1(s+1)$ is then the only maximum element in $x(s+1)$. This then reduces to \emph{Scenario 1}, and following the same analysis, we have $x(s+2)_{max}<x(s+1)_{max}=x_1(s+1)$. Therefore, we finally conclude that $x(s+2)_{max}<x(s)_{max}$. Note that the analysis for $1\notin\mathcal{Q}$ and $2\in\mathcal{Q}$ is exactly the same.

\emph{Scenario 2.C: $1\in\mathcal{Q}$ and $2\in\mathcal{Q}$.}

Again, recall that
\begin{equation*}
\bar{x}_i(s)-\bar{x}_i^2(s)=\sum_{j=1,j\neq i}^n(x_j(s)-x_j^2(s))c_{ji}
\end{equation*}
Based on the analysis in \emph{Scenario 2.B}, we know that $1\in\mathcal{Q}$ and $2\in\mathcal{Q}$ imply the following:
\begin{equation} \label{lemma_decrease_sec2c}
\begin{split}
\text{$c_{21}=1$ and $c_{j1}=0$, for all $j\in\{3,4,\dots,n\}$}\\
\text{$c_{12}=1$ and $c_{k2}=0$, for all $k\in\{3,4,\dots,n\}$}
 \end{split}
 \end{equation}
Transforming condition (\ref{lemma_decrease_sec2c}) into the underlying directed graph represented by $C$, it means that there are only one directed edge from node 2 to node 1 and one directed edge from node 1 to node 2, and no other nodes have directed edges to node 1 or node 2. Therefore, starting from any nodes other than node 1 or node 2, we cannot find a directed path to reach node 1 or node 2, which violates the condition that the directed graph is strongly connected (i.e., $C$ is an irreducible matrix). We hence conclude that \emph{Scenario 2.C} cannot happen.

\emph{Scenario 3: $|\mathcal{M}|=3$.}
In this scenario, there are three maximum elements in $x(s)$. Again, without loss of generality, let $x_1(s)=x_2(s)=x_3(s)=x(s)_{max}$ and we have $x(s)_{max}<0.5$. Now, we can similarly identify four cases.

\emph{Scenario 3.A: $1\notin\mathcal{Q}$, $2\notin\mathcal{Q}$ and $3\notin\mathcal{Q}$.}

We can prove that $x(s+1)_{max}<x(s)_{max}$ by the same reasoning in \emph{Scenario 1}.

\emph{Scenario 3.B: $1\in\mathcal{Q}$, $2\notin\mathcal{Q}$ and $3\notin\mathcal{Q}$.}

Similar to \emph{Scenario 2.B}, we know that
\begin{equation*}
\begin{split}
x(s)_{min}<x_2(s+1)<x(s)_{max}\\
x(s)_{min}<x_3(s+1)<x(s)_{max}
\end{split}
\end{equation*}
Like in \emph{Scenario 2.B}, after one update, this scenario reduces to \emph{Scenario 1} and hence we have $x(s+2)_{max}<x(s)_{max}$. Note that the analysis for either $2\in\mathcal{Q}$ or $3\in\mathcal{Q}$ is the same.

\emph{Scenario 3.C: $1\in\mathcal{Q}$, $2\in\mathcal{Q}$ and $3\notin\mathcal{Q}$.}

Not that the analysis for either $1\notin\mathcal{Q}$ or $2\notin\mathcal{Q}$ is the same as the case when $3\notin\mathcal{Q}$.

From (\ref{lemma_decrease_claim}), and based on a similar analysis in \emph{Scenario 2.B}, we notice that $1\in\mathcal{Q}$ and $2\in\mathcal{Q}$ implies the following conditions:
\begin{equation} \label{lemma_decrease_sec3c}
\begin{split}
\text{$c_{21}\geq 0$, $c_{31}\geq 0$, and $c_{j1}=0$, for all $j\in\{4,5,\dots,n\}$}\\
\text{$c_{12}\geq 0$, $c_{32}\geq 0$, and $c_{k2}=0$, for all $k\in\{4,5,\dots,n\}$}
 \end{split}
 \end{equation}
$C$ is doubly stochastic also implies that $c_{21}+c_{31}=1$ and $c_{12}+c_{32}=1$. Condition (\ref{lemma_decrease_sec3c}) can be further divided into the following three different cases:

\emph{Condition 1: $c_{31}=0$ and $c_{32}=0$}.

Under this situation, condition (\ref{lemma_decrease_sec3c}) reduces to condition (\ref{lemma_decrease_sec2c}), which violates the fact that $C$ is an irreducible matrix. Hence, this situation does not happen.

\emph{Condition 2: either $c_{31}=0$ or $c_{32}=0$, but not both}.

Let us consider $c_{31}=0$, and the analysis for the case when $c_{32}=0$ is exactly the same. Then, we have $c_{21}=1$, $c_{12}>0$ and $c_{32}>0$. From the component-wise Modified DeGroot-Friedkin model (\ref{onestep_component}), we know that
\begin{equation}\label{lemma_decrease_3c_update}
\begin{split}
x_1(s+1)&=x_1^2(s)+x_2(s)-x_2^2(s)\\
x_2(s+1)&=x_2^2(s)+(x_1(s)-x_1^2(s))c_{12}+(x_3(s)-x_3^2(s))c_{32}
\end{split}
\end{equation}
Note that from (\ref{lemma_decrease_3c_update}) and the fact that $1\in\mathcal{Q}$, $2\in\mathcal{Q}$ and $3\notin\mathcal{Q}$, we have $x_1(s+1)=x_2(s+1)=x(s+1)_{max}=x(s)_{max}$ and $x_3(s+1)<x(s)_{max}$. Since $x_3(s+1)$ has reduced, (\ref{lemma_decrease_3c_update}) suggests that after another update, we must have $x_1(s+2)=x(s+1)_{max}=x(s)_{max}$ and $x_2(s+2)<x(s+1)_{max}$. Now, $x_2(s+2)$ has reduced, and it is straightforward to see from (\ref{lemma_decrease_3c_update}) that $x_1(s+3)<x(s+2)_{max}$.

In conclusion, under \emph{Condition 2}, we have $x(s+3)_{max}<x(s)_{max}$.

\emph{Condition 3: $c_{31}>0$ and $c_{32}>0$}.

Similar to the analysis in \emph{Condition 2}, from (\ref{onestep_component}), we have
\begin{equation}\label{lemma_decrease_3c_update_2}
\begin{split}
x_1(s+1)&=x_1^2(s)+(x_2(s)-x_2^2(s))c_{21}+(x_3(s)-x_3^2(s))c_{31}\\
x_2(s+1)&=x_2^2(s)+(x_1(s)-x_1^2(s))c_{12}+(x_3(s)-x_3^2(s))c_{32}
\end{split}
\end{equation}
Then, $x_1(s+1)=x_2(s+1)=x(s+1)_{max}=x(s)_{max}$ and $x_3(s+1)<x(s)_{max}$. Since $x_3(s+1)$ has reduced, (\ref{lemma_decrease_3c_update_2}) suggests that after another update, we must have $x_1(s+2)<x(s+1)_{max}$ and $x_2(s+2)<x(s+1)_{max}$. Hence, we conclude that $x(s+2)_{max}<x(s)_{max}$.

\emph{Scenario 3.D: $1\in\mathcal{Q}$, $2\in\mathcal{Q}$ and $3\in\mathcal{Q}$.}

Under this scenario, we know that the following condition must be satisfied:
\begin{equation} \label{lemma_decrease_sec3c_cond}
\begin{split}
\text{$c_{21}\geq 0$, $c_{31}\geq 0$, and $c_{j1}=0$, for all $j\in\{4,5,\dots,n\}$}\\
\text{$c_{12}\geq 0$, $c_{32}\geq 0$, and $c_{k2}=0$, for all $k\in\{4,5,\dots,n\}$}\\
\text{$c_{13}\geq 0$, $c_{23}\geq 0$, and $c_{l2}=0$, for all $l\in\{4,5,\dots,n\}$}
 \end{split}
 \end{equation}

Again, transforming condition (\ref{lemma_decrease_sec3c_cond}) into the underlying directed graph represented by $C$, it means that we can divide the nodes into two components with component $\#1$ containing nodes $\{1,2,3\}$ and component $\#2$ containing nodes $\{4,5,\dots,n\}$, and further there are no directed edges from component $\#2$ to component $\#1$. This violates the irreducibility assumption of matrix $C$, and hence this case does not happen.

\emph{Scenario 4: $4\leq|\mathcal{M}|\leq n-1$.}

The analysis in \emph{Scenario 3} can be readily applied in this scenario with minor modifications. The difference is that we will have more tedious cases to discuss. The worst case happens in a similar situation as \emph{Scenario 3.C (Condition 2)}, and we can conclude that $x(s+m)_{max}<x(s)_{max}$ if $|\mathcal{M}|=m$.

\emph{Scenario 5: $|\mathcal{M}|=n$.}

This case happens only when $x(s)=\frac{1}{n}\mathbf{1}$, and we know that $\frac{1}{n}\mathbf{1}$ is an equilibrium point.
\end{proof}

Having proved Lemma \ref{lemma_decrease}, we are now in a position to prove Theorem \ref{the4}.

\begin{proof}
\emph{(Theorem \ref{the4})}
Consider the Lyapunov function
$$V(x(s))=x(s)_{max}-\frac{1}{n}.$$
$V(x(s))\geq 0$ for all s. By Lemma \ref{lemma4}, it follows that $V(x(s))$ is a non-increasing function, i.e., $\Delta V(x(s))\le 0$ for all $s$. Furthermore, Lemma \ref{lemma 3} suggests that after finite steps $\tau>0$, $x(\tau)>0$, and we have shown in Lemma \ref{lemma_decrease} that if $x(s)>0$ and $x(s)\neq\frac{1}{n}\mathbf{1}$, then $x(s)_{max}$ must decrease in finite steps. Therefore, $V(x(s))$ must decrease in finite updates, and $\Delta V(x(s))$ cannot be 0 for infinite steps unless $x(s)=\frac{1}{n}\mathbf{1}$, which means that $x(s)$ must converge to the equilibrium point $x^*=\frac{1}{n}\mathbf{1}$. Theorem \ref{the4} is hence proved.
\end{proof}

From Theorem \ref{the4}, self-confidence vector $x(s)$ in the Modified DeGroot-Friedkin model converges to the democratic state $\frac{1}{n}\mathbf{1}$ as in the original DeGroot-Friedkin model for the case of doubly stochastic $C$.
Based on the results, we expect that in general cases, we should get the same stability results.
Simulation results show that for general stochastic $C$, self-confidence vector in the Modified DeGroot-Friedkin model converges to the same nontrivial equilibrium $x^*$ as suggested in Theorem \ref{the1} for the original model. We conjecture that for the Modified DeGroot-Friedkin model, if $C$ is a stochastic matrix that is irreducible with diagonal entries being zero, then there is only one nontrivial equilibrium, which lies in the interior of $\Delta$ and for any $x(0)\in\Delta\setminus\{e_1,\dots,e_n\}$, $x(s)$ converges to this equilibrium point. We leave verification of this conjecture for future work.

\section{Numerical Simulations}\label{simulations}

We first provide some numerical simulations for the cases when the relative interaction matrices $C$ are doubly stochastic to demonstrate the convergence result proved in the last section. Particularly, we consider two networks with five individuals: a directed complete graph and a directed ring graph. The weights for the two graphs are assigned according to the following matrices, respectively.
\begin{eqnarray*}
 C_{\text{complete}}  &=&
  \left[ {\begin{array}{ccccc}
   0 & 0.1 & 0.3 & 0.4 & 0.2\\
   0.6 & 0 & 0.1 & 0.15 & 0.15 \\
    0.2 & 0.3 & 0 & 0.3 & 0.2 \\
     0.1 & 0.35 & 0.1 & 0 & 0.45 \\
     0.1 & 0.25 & 0.5 & 0.15 & 0\\
  \end{array} } \right],\quad \\
  C_{\text{ring}} &=&
  \left[ {\begin{array}{ccccc}
   0 & 1 & 0 & 0 & 0\\
   0 & 0 & 1 & 0 & 0 \\
    0 & 0 & 0 & 1 & 0 \\
     0 & 0 & 0 & 0 & 1 \\
     1 & 0 & 0 & 0 & 0\\
  \end{array} } \right]
\end{eqnarray*}
\begin{figure}
\centering
  \includegraphics[width=0.9\linewidth]{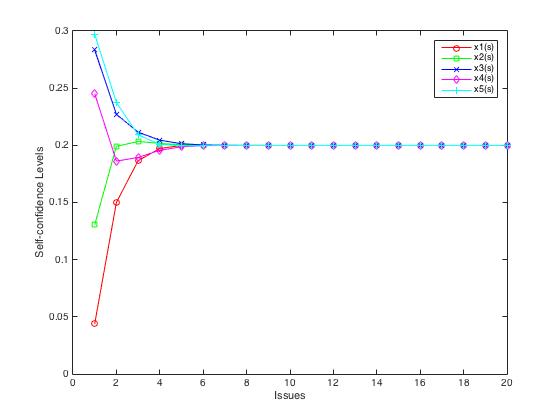}
  \caption{Complete Graph: \\ $x(0)=[0.0439, 0.1305, 0.2834, 0.2452, 0.2970]'$}
  \label{fig:sub1}
  \end{figure}
  \begin{figure}
   \centering

  \includegraphics[width=0.9\linewidth]{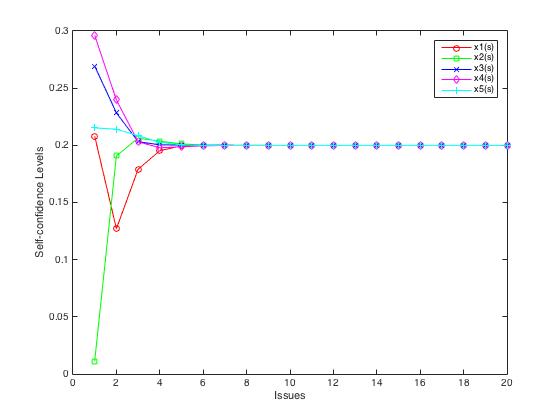}
  \caption{Complete Graph: \\ $x(0)=[0.2080, 0.0113, 0.2693, 0.2962, 0.2152]'$}
  \label{fig:sub2}
\end{figure}
\begin{figure}
\centering
  \includegraphics[width=0.9\linewidth]{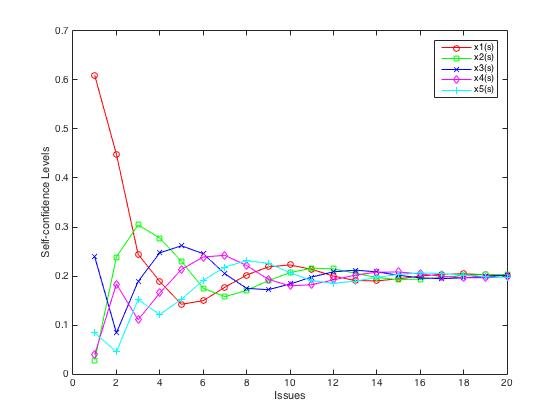}
  \caption{Ring Graph: \\ $x(0)=[0.6097, 0.0275, 0.2391, 0.0399, 0.0838]'$}
  \label{fig:sub3}
  \end{figure}
  \begin{figure}
   \centering

  \includegraphics[width=0.9\linewidth]{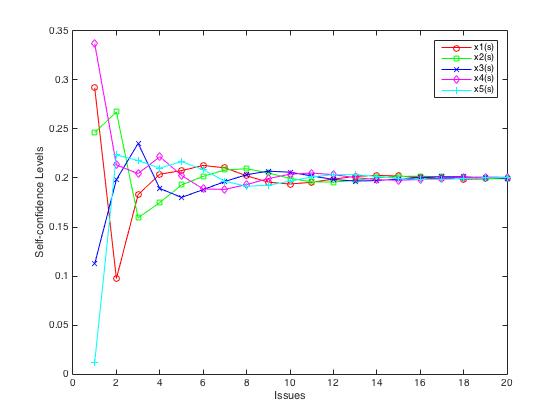}
  \caption{Ring Graph: \\ $x(0)=[0.2920, 0.2464, 0.1124, 0.3370, 0.0122]'$}
  \label{fig:sub4}
\end{figure}

Figure \ref{fig:sub1} and Figure \ref{fig:sub2} are the simulation results for complete graph with different initial conditions, while Figure \ref{fig:sub3} and Figure \ref{fig:sub4} are for the cases of ring graph. As expected, the maximum and the minimum of the self-confidence levels are not expanding, and the self-confidence level for any individual in all cases converges to $0.2$. The convergence in a complete graph only takes five or six issues and is significantly faster than the cases in a ring graph. Complete graph has more edges than a ring graph, so each individual is able to communicate with more other individuals. Instead, in the ring graph, each individual only has one neighbor to interact with. This may be the reason why complete graph has a faster convergence speed.

To further get some insights about our conjecture on the convergence of the cases when the relative interaction matrices $C$ are stochastic but not doubly stochastic, we also provide some simulations for a directed complete graph but with different weights in Figure 5.5 to Figure 5.8. The weights are assigned as follows:
\begin{eqnarray*}
\begin{split}
 C_{\text{complete}}^1 &=&
  \left[ {\begin{array}{ccccc}
   0 & 0.2 & 0.3 & 0.4 & 0.1\\
   0.6 & 0 & 0.1 & 0.15 & 0.15 \\
    0.3 & 0.3 & 0 & 0.3 & 0.1 \\
     0.4 & 0.15 & 0.1 & 0 & 0.35 \\
     0.1 & 0.25 & 0.2 & 0.45 & 0\\
  \end{array} } \right],\\
  C_{\text{complete}}^2 &=&
  \left[ {\begin{array}{ccccc}
   0 & 0.9 & 0.02 & 0.03 & 0.05\\
   0.5 & 0 & 0.3 & 0.1 & 0.1 \\
    0.25 & 0.25 & 0 & 0.2 & 0.3 \\
     0.7 & 0.1 & 0.05 & 0 & 0.15 \\
     0.35 & 0.25 & 0.25 & 0.15 & 0\\
  \end{array} } \right]
  \end{split}
\end{eqnarray*}
\begin{figure}
\centering
  \includegraphics[width=0.9\linewidth]{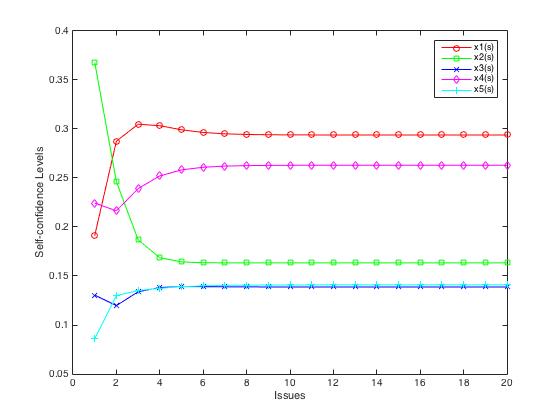}
  \caption{$C^1_{\text{complete}}$: \\ $x(0)=[0.1911, 0.3681, 0.1305, 0.2245, 0.0858]'$}
  \label{fig:sub5}
  \end{figure}
  \begin{figure}
   \centering

  \includegraphics[width=0.9\linewidth]{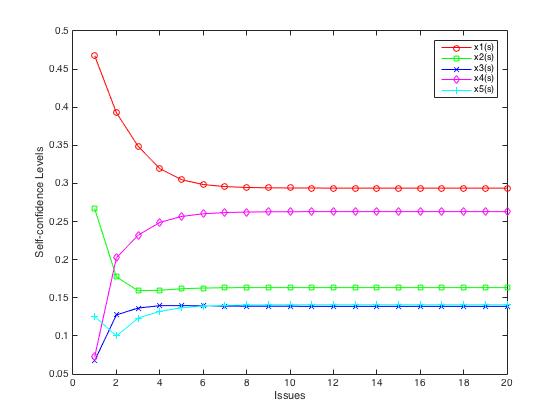}
  \caption{$C^1_{\text{complete}}$: \\ $x(0)=[0.4675, 0.2667, 0.0676, 0.0727, 0.1255]'$}
  \label{fig:sub6}
\end{figure}
\begin{figure}
\centering
  \includegraphics[width=0.9\linewidth]{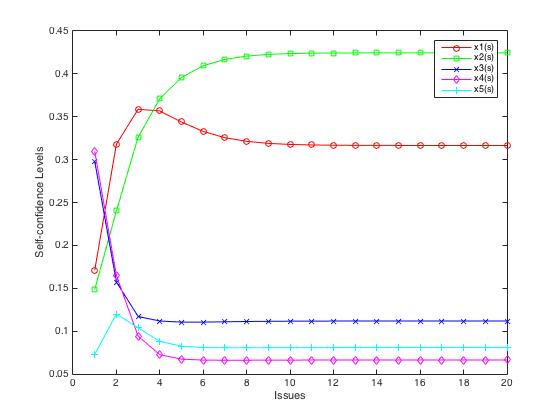}
  \caption{$C^2_{\text{complete}}$: \\ $x(0)=[0.1709, 0.1486, 0.2981, 0.3097, 0.0728]'$}
  \label{fig:sub7}
  \end{figure}
  \begin{figure}
   \centering

  \includegraphics[width=0.9\linewidth]{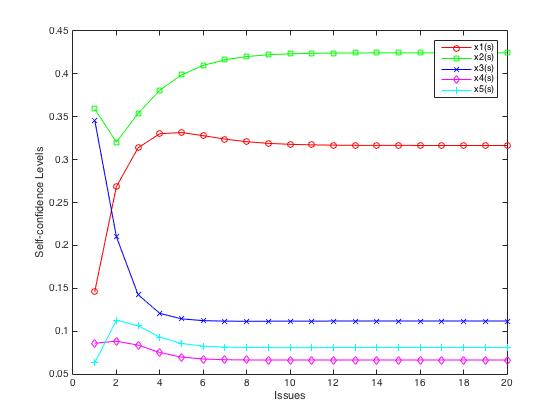}
  \caption{$C^2_{\text{complete}}$: \\ $x(0)=[0.1459, 0.3592, 0.3462, 0.0859, 0.0628]'$}
  \label{fig:sub8}
\end{figure}

As we have conjectured, given a stochastic and irreducible relative interaction matrix $C$, the self-confidence levels converge, but the convergent values depend on the specific weights in $C$, i.e., for different $C$, the convergent self-confidence vectors are different. Furthermore, note that we no longer have the nice property that the maximum and the minimum of the self-confidence levels are not expanding. In fact, as simulations suggest, it is quite possible that the maximum is increasing or the minimum is decreasing. Therefore, the proof techniques used in this paper cannot apply in such situation, and we leave the proof of our general conjecture for future work.

\section{Conclusions} \label{sec_conclusion}
In this paper, we have introduced a Modified DeGroot-Friedkin model, which allows individuals to update their self-confidence levels after each discussion on a particular issue. We have then investigated the limiting behaviors of the self-confidence vector when a sequence of issues are discussed. A complete analysis for the case when the underlying interaction matrix is doubly stochastic has been provided. As expected, the self-confidence vector converges to the equal-weights vector, meaning that eventually each individual will have the same level of self-confidence.

This paper serves as a starting point for this line of research and many questions still remain to be answered. As we have seen, the stability of the modified model and the original DeGroot-Friedkin model coincides for the case of doubly stochastic interaction matrix, which suggests that there might be similar connections for more general settings. A future direction that is of particular interest is to mathematically characterize the properties of the Modified DeGroot-Friedkin model for general stochastic interaction matrices under the condition that $x(0)\in \Delta$ or $x(0)$ is not necesarrily in $\Delta$. In addition, future work will focus on more general finite-steps cases (\ref{finite_step}), i.e., $x_i(s+1)=p_i(s,T)$ for finite $T>1$. We expect that there would be some similar behaviors when we go from one-step to finite-steps.
\addtolength{\textheight}{-12cm}   

\section*{Acknowledgment}

The authors are grateful to Mohamed Ali Belabbas, Ali Khanafer,
and Meiyue Shao (Lawrence Berkeley National Laboratory) for insightful discussions,
and to Xudong Chen for suggesting a fix to an error in an earlier proof of Theorem \ref{the4}.




\bibliographystyle{unsrt}
\bibliography{ACC2015}




\end{document}